\documentclass[12pt]{amsart}
\usepackage{amsmath,amsthm,amssymb,amsfonts}

\usepackage{enumerate}
\usepackage{nicematrix}
\usepackage{enumitem} 
\usepackage{graphicx}
\usepackage{hyperref}
\usepackage{tikz}
\allowdisplaybreaks
 \newtheorem{theorem}{Theorem}[section]

 \newtheorem{corollary}[theorem]{Corollary}

 \newtheorem{conjecture}[theorem]{Conjecture}
  \newtheorem{lemma}[theorem]{Lemma}
\newtheorem{example}[theorem]{Example}
\theoremstyle{remark}

\newtheorem{fact*}{Fact}

\DeclareMathOperator{\IM}{Im}

\newcommand{\hilbert}{\mathcal{H}}

\newcommand{\A}{\mathcal{A}}

\newcommand{\C}{\mathbb{C}}

\newcommand{\R}{\mathbb{R}}

\newcommand{\cc}[1]{\overline{#1}}
\newcommand{\abs}[1]{\left\vert#1\right\vert}

\newcommand{\ip}[2]{\left\langle #1, #2 \right\rangle}

\newcommand{\inv}{^{-1}}

\newcommand{\til}{\raise.17ex\hbox{$\scriptstyle\mathtt{\sim}$}}

\newcommand\la{\lambda}

\newcommand\beq{\begin{equation}}

\newcommand\eeq{\end{equation}}

\newcommand\black{\color{black}}

\newcommand\bbm{\begin{bmatrix}}
\newcommand\ebm{\end{bmatrix}}
\newcommand{\bpm}{\left( \begin{smallmatrix}}
\newcommand{\epm}{\end{smallmatrix} \right)}
\numberwithin{equation}{section}

\newlength{\Mheight}
\newlength{\cwidth}

\newcommand{\dfn}[1]{{\bf #1}\index{#1}}

\newcommand{\cA}{\mathcal{A}}

\newcommand{\comb}{\triangleright}

\title{Cauchy transforms of colored graphs in two variables}
\author[L. Adlin]{Lily Adlin}
\author[G. Thai]{Giovani Thai}
\author[S. Tiscareno]{Samuel Tiscareno}
\author[R. Tully-Doyle]{
Ryan Tully-Doyle$^\dagger$
}
\address{Mathematics Department \\
Cal Poly, SLO\\
1 Grand Ave \\
San Luis Obispo, CA 93407}
\email[R. Tully-Doyle]{rtullydo@calpoly.edu}
\thanks{$\dagger$ Partially supported by National Science Foundation DMS Analysis Grant 2055098}


\begin{document}

\begin{abstract}
By designating vertices with variables, a simple undirected graph can be augmented to have associated representing rational function in two variables taking the complex bi-upper halfplane to itself. We give relations between representing functions of certain products of such graphs by way of Schur complements. We also study the connection between the structure of the graph and the regularity of the representing function at a boundary singularity.
\end{abstract}

\maketitle

\section{Introduction}
Let $\Pi = \{z \in \C: \IM z > 0\}$ denote the complex upper halfplane. A Pick function is an analytic function $f$ with the property that $\IM f(z) \geq 0$ whenever $z \in \Pi$. (We note that these functions are referred to in parallel literature as Herglotz, Nevanlinna, or R- functions.) An inner Pick function has the additional property that $f(z) \in \R$ whenever $z \in \R$ (that is, inner Pick functions map the boundary into the boundary. With a mild assumption on the growth of $f$ along the imaginary axis, such functions have what is known as a \dfn{Nevanlinna representation} - that is, such $f$ are the Cauchy transforms of positive Borel measures on $\R$.

In \cite{aty12}, Agler, Tully-Doyle, and Young gave a generalization to several variables. Here, we give the two variable case. In two variables, a Pick function is analytic function $f: \Pi^2 \to \cc{\Pi}$. An inner Pick function additionally has $f(z) \in \R$ whenever $(z,w) \in \R^2$.

\begin{theorem}
Suppose that $f$ is a Pick function  in two variables. If $\displaystyle \lim_{s\to\infty} s \abs{f(is, is)} < \infty$, then there exist a separable Hilbert space $\hilbert$, a (densely-defined) self-adjoint operator $A$ on $\hilbert$, a positive contraction $Y$ on $\hilbert$, and a vector $\alpha \in \hilbert$ so that 
\[
f(z, w) = \ip{(A - z_Y)\inv\alpha}{\alpha}_\hilbert,
\]
where $z_Y = zY + w(I - Y)$.
\end{theorem}

In the converse direction, Pascoe \cite{pascoe1}, Bickel, Pascoe, and Sola \cite{bps2}, and Bickel and Hong \cite{bickelhong} gave specific classes of $A$, $Y$ and $\alpha$ guaranteed to produce a rational inner Pick function with a boundary singularity.

\begin{theorem}[\cite{pascoe1, bps2, bickelhong}]\label{blackmagic}
Let $A_G$ be the adjacency matrix of a simple undirected graph $G$ on $n$ vertices. Let $Y$ be the diagonal $n\times n$ matrix 
\beq\label{Ydef}
Y = \bbm I_{n-1} & 0 \\ 0 & t \ebm
\eeq
where $t \in [0,1]$. Let $e_1$ be the first standard basis vector in $\C^n$. Then
\[
f_G(z, w) = \ip{(A_G - z_Y)\inv e_1}{e_1}
\]
is a rational inner Pick function with a (regular) boundary singularity at $(0, 0)$.
\end{theorem}

Given such a graph $G$, we call the resulting Pick function $f_G$ the \emph{representing function} of $G$. These sorts of Cauchy transforms of graph related objects have been studied in the context of free products of graphs in, e.g, \cite{quenell, gutkin, carter}. Other work related to asymptotics and cumulants can be found, for example, in \cite{agax, ahl22}.

The graph theoretic viewpoint taken by Bickel and Hong in \cite{hong, bickelhong} provides a new vantage on the structure of rational inner functions. Following them, our work is concerned with two basic questions. First, can we recover information about the boundary regularity of $f_G$ from the structure of $G$? This question is deeply related to ongoing study of rational inner functions on the bidisk initiated by Bickel, Pascoe, and Sola \cite{bps1, bps2}, in particular the notion of \emph{contact order} at a boundary singularity. Second, what is the connection between certain products or decomposition of graphs and the relationship between their representing functions? Here, we follow work that recently originated in the thesis of Hong \cite{hong} and was followed up by Bickel and Hong \cite{bickelhong}.  Given the adjacency matrix $A$ of a graph $G$, we can view the expression $A - z_Y = A - zY - w(1-Y)$ with $Y$ as in \eqref{Ydef} as the adjacency matrix of a chromatic graph with self-loops at each vertex weighted by $z$ and $w$. Viewing these self-loops as a sort of vertex coloring, when $t = 0$ in $Y$, such a graph will have exactly one vertex colored with a $w$ and all others colored with $z$.

Our main results are as follows. Given two rooted simple undirected graphs $G, G'$, the \emph{star product} (to use the terminology of \cite{ahl22}) $G \star G'$ is the graph obtained by joining $G$ and $G'$ at the root. For compatible graphs, we show in Corollary \ref{staridentity} by elementary linear algebraic arguments an additive relationship between the representing functions. Likewise, the \emph{comb product} $G \comb G'$ attaches a copy of $G'$ to every vertex of $G$. For compatible graphs, we show a composition relationship between the representing functions of the graph in Corollary \ref{combidentity}. (Such compositions are typically quite difficult to recover for even rational two-variable functions.) Finally, in Theorem \ref{contactorder}, we connect the graph theory to the boundary behavior of associated functions; we prove that given a colored adjacency matrix with a single vertex colored $w$ and all others $z$,  the path length connecting the $z$-colored root to the single $w$-colored vertex recovers the boundary regularity (in terms of contact order) of $f_G$ at infinity, partially resolving a conjecture of Bickel and Hong \cite{bickelhong} and giving a method for constructing a rich class of examples with prescribed boundary behavior.

\section{Preliminaries}

\subsection{Nevanlinna representations}

The following basic lemma is essentially contained in \cite{aty12} in the discussion of type I operator resolvents.

\begin{lemma}[\cite{aty12} Proposition 2.2 ]\label{nevproperties}
    Let $A = A^*$; $Y = \begin{bmatrix} 
t_{1} & & \\
& \ddots & \\
& & t_{n} \\
\end{bmatrix}$ where $0 \leq t_{i} \leq 1$; and $\alpha \in \mathbb{C}^n$. Then $(A - z_Y)\inv$ is well-defined on $\Pi^2$. Moreover, for all $(z,w) \in \Pi^2$,
\[
\IM (A - z_Y)\inv \geq 0.
\]

\end{lemma}

Lemma \ref{nevproperties} gives a useful version of Theorem \ref{blackmagic}.

\begin{theorem}\label{blackmagicbetter}
Let $A_G$ be the adjacency matrix of a simple undirected graph $G$ on $n$ vertices. Let $Y$ be the diagonal $n\times n$ matrix 
\beq\label{Ydef2}
Y = \bbm I_{k} & 0 \\ 0 & 0 \ebm.
\eeq
Let $e_j$ be the $j$-th basis vector in $\C^n$. Then
\[
f(z, w) = \ip{(A_G - z_Y)\inv e_j}{e_j}
\]
is a rational inner Pick function with a (regular) boundary singularity at $(0,0)$.
\end{theorem}

\subsection{Graph constructions}

 From a simple undirected graph $G$, construct the \emph{$(z,w)$-colored graph} (or simply \emph{colored graph}) by adding to $G$ a self loop with the weight equal to the negative of the label. The adjacency matrix for this modified graph we denote $$\mathcal{A}_G = A_G - zY - w(I-Y),$$ where $$
Y = \bbm I_{n-k} & 0 \\ 0 & 0 \ebm 
$$  with the entries of 1 on the diagonal corresponding to the vertices tagged with $z$. See Figure \ref{fig:analog} representing the colored graph $\cA_G$ with adjacency matrix 
\[
\cA_G = \bbm 0 & 1 \\ 1 & 0 \ebm - z\bbm 1 & 0 \\ 0 & 0 \ebm - w \bbm 0 & 0 \\ 0 & 1 \ebm = \bbm -z & 1 \\ 1 & -w \ebm.
\]

\begin{figure}[!htbp]
    \centering
    \begin{tikzpicture} [scale=.6,auto=center, every node/.style={circle,fill=blue!20}]
    \node[label=below:$z$] (1) at (0, 0) {};
    \node[label=below:$w$] (2) at (4, 0) {};

    \node[fill=white] (eq) at (6, 0) {=};

    \node (3) at (9, 0) {};
    \node (4) at (13, 0) {};
    
    \draw (1) -- (2);
    \draw (3) -- (4);
    \path (3) edge [loop left] node[fill=white] {$-z$} (3);
    \path (4) edge [loop right] node[fill=white] {$-w$} (4);
    \end{tikzpicture}
    \caption{Notation for coloring vertices, describes adding self-loops with weights $-z$ or $-w$.}
    \label{fig:analog}
\end{figure}
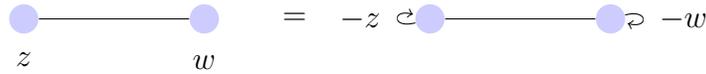

\subsection{Nevanlinna representation of graphs}

We begin with a graph theoretic interpretation of the Nevanlinna representation in Theorem \ref{blackmagic}. The first observation is that viewing the kernel of the Nevanlinna representation in Theorem \ref{blackmagic} as the inverse of the adjacency matrix of the colored graph $\cA_G$, we have
\[
(A - z_Y)\inv = \cA_G\inv.
\]

\begin{lemma}
Let $\cA_G$ be the adjacency matrix of a $(z,w)$-colored graph $G$ on $n$ vertices, and let $e_k$ be the $k$th standard basis vector in $\C^k$. Then
\[
f_G(z,w) = \ip{\cA\inv_G e_k}{e_k}
\]
is a rational inner Pick function.
\end{lemma} 
\begin{proof}
This is essentially a restatement of Theorem \ref{blackmagicbetter} under permutation.
\end{proof}

That is, given $\cA_G$, we can compute a representing function of $G$ with respect to vertex $k$ by using $e_k$ in the inner product. This function is the representing function of $G$ at vertex $k$, denoted $f_G^k(z,w)$. We also introduce the notation
\[
\ip{ \cA_G\inv e_k}{e_k} = (\cA_G\inv)_{(k,k)} = f_G^k(z,w),
\]
so that $f_G^1 = f_G$.

Representing functions for graphs are invariant under relabeling of vertices \cite[Lemma 2.3]{bickelhong}. We take a graph theoretic approach.

\begin{lemma}\label{relabel}
    Let $G$ be a graph with colored adjacency matrix $\cA_G$ and $H$ be a graph with colored adjacency matrix $\mathcal{A}_H$. If $H = \phi(G)$ with graph isomorphism $\phi$ then 
    $$f^k_G(z,w) = f^{\phi(k)}_H(z,w).$$
\end{lemma}

\begin{proof}
    Note that $f_G(z,w) = (\cA_G^{-1})_k$. Since $\phi$ is an isomorphism, there exists a permutation matrix $U_\phi$ so that $U_{\phi}\inv A_G U_\phi = A_H$. 
 
    Then 
    \begin{align*}
    f^{\phi(k)}_H(z,w) &= (\cA_H\inv)_{(\phi(k),\phi(k))} \\
    &= (U_{\phi}^{-1} \cA_G U_{\phi})^{-1}_{(\phi(k),\phi(k))} \\
    &= (U_{\phi}^{-1}\cA_G^{-1}U_{\phi})_{(\phi(k),\phi(k))}\\
    &= (\cA_G\inv)_{(k,k)} \\
    &= f_G^k(z,w).
    \end{align*}
\end{proof}

Note that this implies that we can put any colored adjacency matrix into the form required to apply Theorem \ref{blackmagicbetter} (essentially collecting the $z$'s together). Essentially, we can relabel the vertices at will. Most often, we compute with respect to vertex $1$ and suppress the additional superscript.

Further, adding new components to a graph does not change the representing function (essentially, walks emanating from a vertex in one component cannot travel into other components). A related result is in \cite{bickelhong} as Proposition 2.4 with a different argument.

\begin{lemma} \label{connectedcomponent}
    Let $G$ be a graph with colored adjacency matrix $\cA_G$. Let $K$ be a graph formed from $G$ and a component $H$ disconnected from $G$. Then $f_G(z,w) = f_K(z,w)$. 
\end{lemma}

\begin{proof}
    Let $\cA_H$ be the colored adjacency matrix of $H$. Because $G$ and $H$ are disconnected, we can choose some labeling of the vertices so that $$\cA_K = \begin{bmatrix}
        \cA_G & 0 \\
        0 & \cA_H \\
    \end{bmatrix}.$$ Thus,
    \begin{align*}
        f_K(z,w) &= (\cA_K^{-1})_{(1,1)} \\
        &= \begin{bmatrix}
        \cA_G & 0 \\
        0 & \cA_H \\
    \end{bmatrix}^{-1}_{(1,1)} \\
    &= \begin{bmatrix}
        \cA_G^{-1} & 0 \\
        0 & \cA_H^{-1} \\
    \end{bmatrix}_{(1,1)} \\
    &= (\cA_G^{-1})_{(1,1)} \\
    &= f_G(z,w). \\
    \end{align*}
\end{proof}

As we will be using Schur complements, we need the following observation about the blocks of colored adjacency matrices. 
\begin{lemma}\label{invertible}
If $(z, w) \in \Pi^2$ and 
\[
\cA_G = \bbm a_{11} & A_{12} \\ A_{12}^T & A_{22} \ebm
\]
then $A_{22}$ is invertible.

\end{lemma}

\begin{proof}
The block $A_{22}$ can be viewed as colored adjacency matrix, and so by Lemma \ref{nevproperties} is invertible.
\end{proof}

\section{Graph products and representing functions}

We now investigate the relationship between representing functions induced by taking certain  graph products. 

The results in this chapter are given in the one-variable case in \cite{ahl22} as Proposition 2.1 and Proposition 2.5, where they are given without proof (essentially, they are special cases of much more general constructions in \cite{Speicher1993BooleanC, Popa2008ANP, muraki}). In the following sections, we generalize these results to two variables via Schur complement arguments.

Note that the identities proven in this section will involve the reciprocal of our representing function, sometimes called the reciprocal Cauchy transform \cite{ahl22}. As a notational convenience, we denote
\[
g_G(z,w) = \frac{1}{f_G(z,w)}.
\]

\subsection{Star Products}

Let $G$ and $G'$ be rooted graphs with corresponding roots $r_1, r_2$ and $z, w$ vertex coloring. The \dfn{star product} of $G$ and $G'$, denoted $G \star G'$, attaches $G$ and $G'$ at their root vertices (Figure \ref{fig:starprod}).
\begin{figure}[!htbp]
    \centering
    \begin{tikzpicture} [scale=.7,auto=center, every node/.style={circle,fill=blue!20}]
    \node (g1) at (0, 0) {};
    \node (g2) at (2, 0) {};
    \node[fill=red] (g3) at (2, 2) {};
    \node (g4) at (0, 2) {};

    \node[fill=red] (h1) at (4, 0) {};
    \node (h2) at (6, 0) {};
    \node (h3) at (5, 2) {};

    \node[fill=white, scale=1] (s) at (3, 1) {$\star$};
    \node[fill=white] (eq) at (7, 1) {=};

    \node (1) at (8, 1) {};
    \node (2) at (9, 0) {};
    \node[fill=red] (3) at (10, 1) {};
    \node (4) at (9, 2) {};
    \node (5) at (12, 2) {};
    \node (6) at (12, 0) {};

    \draw (g1) -- (g2);
    \draw (g2) -- (g3);
    \draw (g3) -- (g4);
    \draw (g4) -- (g1);

    \draw (h1) -- (h2);
    \draw (h2) -- (h3);
    \draw (h3) -- (h1);

    \draw (1) -- (2);
    \draw (2) -- (3);
    \draw (3) -- (4);
    \draw (4) -- (1);
    \draw (3) -- (5);
    \draw (5) -- (6);
    \draw (6) -- (3);
    
    \end{tikzpicture}
    \caption{Example of a star product of two graphs.}
    \label{fig:starprod}
\end{figure}
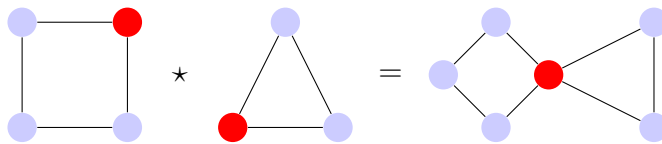

\begin{theorem}
Let $G$ and $H$ be rooted graphs with a shared color of $a_{11}$ on their root vertices. Define the zero graph to be a single vertex with color $a_{11}$. Then the Nevanlinna representation of $G \star H$ with respect to the root vertex satisfies the identity
\[
g_{G\star H}(z,w) = g_G(z,w) + g_{H}(z,w) - g_0(z,w).
\]

\end{theorem}

\begin{proof}
Let $\cA_G = \begin{bmatrix}
-a_{11} & A_{12} \\
A_{12}^T & A_{22} 
\end{bmatrix}$  be the colored adjacency matrix for $G$ and $\cA_H = \begin{bmatrix}
-a_{11} & B_{12} \\
B_{12}^T & B_{22} 
\end{bmatrix}$ be the colored adjacency matrix of $H$, with both graphs arranged so that the root is labeled the first vertex. The graph $0$ has colored adjacency matrix $\begin{bmatrix}
    -a_{11}\\
\end{bmatrix}$.

 So by Schur complements, \begin{align*}
    g_{G}(z,w) &=\begin{bmatrix}
-a_{11} & A_{12} \\
A_{12}^T & A_{22} 
\end{bmatrix}\inv_{(1,1)} = -a_{11}\inv - A_{12}A_{22}^{-1}A_{12}^T \\
    g_{H}(z,w) &= \begin{bmatrix}
-a_{11} & B_{12} \\
B_{12}^T & B_{22} 
\end{bmatrix}\inv_{(1,1)} =  -a_{11}\inv - B_{12}B_{22}^{-1}B_{12}^T \\
    g_0(z,w) &= -a_{11}\inv
\end{align*}

Note that $G \star H$ has colored adjacency matrix
$\cA_{G \star H} = \begin{bmatrix}
    \cA_G &\begin{pmatrix}
    B_{12}\\
    0
    \end{pmatrix} \\
    \begin{pmatrix}
    B_{12}^T &
    0
    \end{pmatrix} & B_{22}
\end{bmatrix}$.  Since $f_{G \star H}(z,w) = (\cA_{G \star H}^{-1})_{(1,1)}$, we can invoke Lemma \ref{invertible} and compute this by the Schur complement, which gives
\begin{align*}
     f_{G \star H}(z,w)  &= \begin{bmatrix}
    \cA_G &\begin{pmatrix}
    B_{12}\\
    0
    \end{pmatrix} \\
    \begin{pmatrix}
    B_{12}^T &
    0
    \end{pmatrix} & B_{22}
\end{bmatrix}\inv_{(1,1)} \\
&=\left(\cA_G - \begin{bmatrix}
    B_{12}\\0
\end{bmatrix}B_{22}^{-1}\begin{bmatrix}
    B_{12}^T & 0
\end{bmatrix}\right)^{-1}_{(1,1)} \\
&= \left(\begin{bmatrix}
-a_{11} & A_{12} \\
A_{12}^T & A_{22} 
\end{bmatrix} - 
\begin{bmatrix}
    B_{12}B_{22}^{-1}B_{12}^T & 0 \\
    0 & 0
\end{bmatrix}\right)^{-1}_{(1,1)} \\
&= \begin{bmatrix}
     g_{H}(z,w) & A_{12}\\
    A_{12}^T & A_{22}
\end{bmatrix}^{-1}_{(1,1)}\\
&= (g_{H}(z,w) - A_{12}A_{22}^{-1}A_{12}^T)^{-1}\\
&=(g_{H}(z,w) + g_G(z,w) + a_{11})^{-1} \\
\end{align*} 

As these quantities are scalar, taking the inverse of both sides and substituting $g_0(z,w)$ for $-a_{11}$  gives us
\[
    g_{G\star H}(z,w) = g_G(z,w) + g_{H}(z,w) - g_0(z,w).
\]
\end{proof}

From a graph theory perspective, the star product is a reciprocal sum, but one must subtract the doubled self-loop introduced at the combining vertex. Note that this represents a type of ``generalized resistor law'' result, as
\[
\frac{1}{f_{G \star G'}} = \frac{1}{f_G} + \frac{1}{f_{G'}} -z.
\]

\begin{example}
Consider the graphs $G$ and $H$ from Figure \ref{fig:starprod} and suppose that $\A_G$ and $\A_H$ are given by 
\begin{align*}
\A_G &= \left[\begin{smallmatrix} -z & 1 & 0 & 1 \\ 1 & -w & 1 & 0 \\ 0 & 1 & -w & 1 \\ 1 & 0 & 1 & -w \end{smallmatrix}\right] \\
\A_H &= \left[\begin{smallmatrix} -z & 1 & 1 \\ 1 & -z & 1 \\ 1 & 1 & -w \end{smallmatrix}\right].
\end{align*}
Then
\[
\A_{G \star H} = \left[ \begin{smallmatrix} -z & 1 & 0 & 1 & 1 & 1 \\ 1 & -w & 1 & 0 & 0 & 0  \\ 0 & 1 & -w & 1 & 0 & 0 \\ 1 & 0 & 1 & -w & 0 & 0 \\ 1 & 0 & 0 & 0 & -z & 1 \\ 1 & 0 & 0 & 0 & 1 & -w \end{smallmatrix} \right].
\]
Calculating, we get
\begin{align*}
g_G + g_H &= \frac{w^3 z-2 w^2-2 w z}{2 w-w^3} + \frac{-w z^2+w+2 z+2}{w z-1} \\
&= \frac{-2 w^3 z^2+w^3+5 w^2 z+2 w^2+4 w z^2-4 w-6 z-4}{\left(w^2-2\right) (w z-1)} \\
&= g_{G\star H} - z.
\end{align*}
\end{example}

Graphs that decompose as star products also have an interpretation as a more general class of colored graphs, where we recolor vertices with functions that represent graph structures.

\begin{theorem}\label{thm_retract}
    Let G be a graph that can be written as $G = H \star K$ connected only at vertex i. Then $f_G(z,w) = f_{H'}(z,w)$, where $\A_{H'}$ represents the same graph as $\A_H$ except that vertex $i$ is colored with $-g_K(z,w)$. 
\end{theorem}

\begin{proof}
    Without loss of generality, suppose that we calculate the Nevanlinna representation with respect to $e_1$. Let $H$ have colored adjacency matrix $\A_H$, a block matrix that can be written as $$\A_H = \begin{bmatrix}
        A_{11} & A_{1i} & A_{12} \\
        A_{1i}^T & -a_{ii} & A_{i2} \\
        A_{12}^T & A_{i2}^T & A_{22}
    \end{bmatrix}$$ Similarly, $$\A_K = \begin{bmatrix}
        -a_{ii} & B_{12} \\
        B_{12}^T & B_{22}
    \end{bmatrix}$$  is the colored adjacency matrix for $K$ where vertex 1 is the shared root in the star product decomposition. As the graphs are connected at the single vertex $i$, $G$ has colored adjacency matrix 
    \begin{align*}    
    \A_G &= \begin{bmatrix}
        A_{11} & A_{1i} & A_{12} & 0 \\
        A_{1i}^T & -a_{ii} & A_{i2} & B_{12}\\
        A_{12}^T & A_{i2}^T & A_{22} & 0 \\
        0 & B_{12}^T & 0 & B_{22}
    \end{bmatrix} 
    \end{align*}
    So, applying Schur complements, and noting that in the scalar case that matrix inverse and reciprocals coincide,
    \begin{align*}
    f_G(z,w) &= (\A_G^{-1})_{(1,1)}\\
    &= \left(\A_H - 
    \begin{bmatrix}
        0 \\ B_{12} \\ 0
    \end{bmatrix}B_{22}^{-1}
    \begin{bmatrix}
      0 & B_{12}^T & 0  
    \end{bmatrix}\right)
    ^{-1}_{(1,1)} \\
    &= 
    \left(\begin{bmatrix}
        A_{11} & A_{1i} & A_{12} \\
        A_{1i}^T & -a_{ii} & A_{i2} \\
        A_{12}^T & A_{i2}^T & A_{22}
    \end{bmatrix} - 
    \begin{bmatrix}
        0 & 0 & 0 \\
        0 & B_{12}B_{22}^{-1}B_{12}^T & 0 \\
        0 & 0 & 0
    \end{bmatrix}\right)
    ^{-1}_{(1,1)} \\
    &=
    \begin{bmatrix}
        A_{11} & A_{1i} & A_{12} \\
        A_{1i}^T & g_K(z,w) & A_{i2} \\
        A_{12}^T & A_{i2}^T & A_{22}
    \end{bmatrix}
    ^{-1}_{(1,1)}\\
    &= f_{H'}(z,w).
    \end{align*}
\end{proof}
\black

This means that structures in graphs that are connected to the rest of the graph by a single vertex can be reduced down and fed into the representing function as essentially the new color of that vertex. The result of this is that loops are the irreducible structures of graphs with respect to this representation. One can view this as morally an ``equivalent resistor'' replacement, and suggests that significantly more complicated graph reductions are possible with commensurately more complicated linear algebra.

\begin{example}
\begin{figure}[!htbp]
    \centering
    \begin{tikzpicture} [scale=.7,auto=center, every node/.style={circle,fill=blue!20}]

    \node[label=left: {$z$}] (1) at (0, 1) {};
    \node[label=below:{$z$}] (2) at (1, 0) {};
    \node[fill=white] (H) at (1, -2) {$H$};
    \node[fill=white] (H) at (2.5, -2) {$*$};
    \node[fill=white] (H) at (4, -2) {$K$};
    \node[fill=red, label=below:{$z$}] (3) at (2, 1) {};
    \node[label=above:{$w$}] (4) at (1, 2) {};
    \node[label=above:{$z$}] (5) at (4, 2) {};
    \node[label=below:{$w$}] (6) at (4, 0) {};
    \node[fill=white] (eq) at (5,1) {$\sim$};

        \node[label=left: {$z$}] (7) at (7, 1) {};
    \node[label=below:{$z$}] (8) at (8, 0) {};
    \node[fill=red, label=right:{$-g_K$}] (9) at (9, 1) {};
    \node[label=above:{$w$}] (10) at (8, 2) {};

    \draw (1) -- (2);
    \draw (2) -- (3);
    \draw (3) -- (4);
    \draw (4) -- (1);
    \draw (3) -- (5);
    \draw (5) -- (6);
    \draw (6) -- (3);
    \draw (7) -- (8);
    \draw (8) -- (9);
    \draw (9) -- (10);
    \draw (10) -- (7);

    \end{tikzpicture}
    \caption{Example of a colored graph reduction.}
    \label{fig:reduction}

\end{figure}
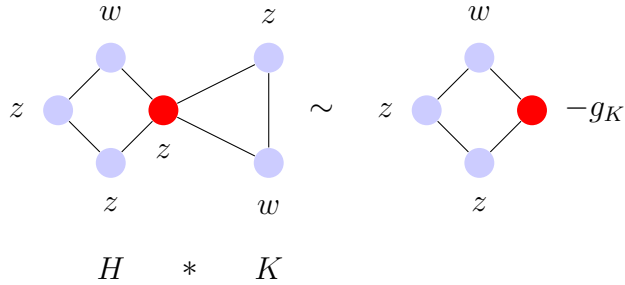

In Figure \ref{fig:reduction}, we give a graphical example of the matrix relation 
\[
\bbm  -z & 1 & 1 & 0 & 0 & 0 \\ 1 & -w & 0 & 1 & 0 & 0\\ 1 & 0 & -z & 1 & 0 & 0 \\ 0 & 1 & 1 & -z & 1 & 1 \\ 0 & 0 & 0 & 1 & -z & 1 \\ 0 & 0 & 0 & 1 & 1 & -w \ebm\inv_{(1,1)} = \bbm -z & 1 & 1 & 0 \\ 1 & -w & 0 & 1 \\ 1 & 0 & -z & 1 \\ 0 & 1 & 1 & g_H \ebm\inv_{(1,1)}
\]
where 
\[
g_K = 1/f_K =1/ \bbm -z & 1 & 1 \\ 1 & -z & 1 \\ 1 & 1 & -w \ebm\inv_{(1,1)} = \frac{-w z^2+w+2 z+2}{w z-1}.
\]
In either case, the representing function for the colored graphs is
\[
f_G(z, w) = \frac{-w^2 z^3+2 w^2 z+3 w z^2+2 w z-w-z}{w^2 z^4-3 w^2 z^2+w^2-4 w z^3-2 w z^2+4 w z+2 w+3 z^2+2 z}.
\]
\black
\end{example}

\subsection{Comb Products}

Let $G$ be a graph and $G'$ be a rooted graph. The \dfn{comb product} of $G$ and $G'$ is the star product of $G$ with $G'$, using every vertex of $G$ as a root; see Figure \ref{fig:combprod}.

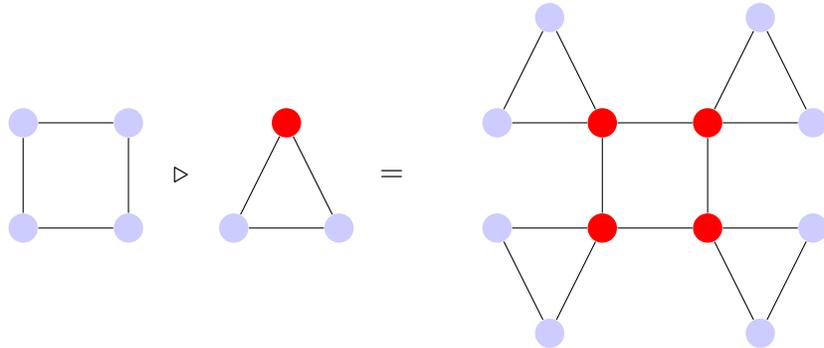
\begin{figure}[!htbp]
    \centering
    \begin{tikzpicture} [scale=.7,auto=center, every node/.style={circle,fill=blue!20}]
    \node (g1) at (0, 0) {};
    \node (g2) at (2, 0) {};
    \node (g3) at (2, 2) {};
    \node (g4) at (0, 2) {};

    \node (h1) at (4, 0) {};
    \node (h2) at (6, 0) {};
    \node[fill=red] (h3) at (5, 2) {};

    \node[fill=white, scale = 1] (c) at (3, 1) {$\comb$};
    \node[fill=white] (eq) at (7, 1) {=};

    \node[fill=red] (1) at (11, 0) {};
    \node[fill=red] (2) at (13, 0) {};
    \node[fill=red] (3) at (13, 2) {};
    \node[fill=red] (4) at (11, 2) {};
    \node (11) at (9, 0) {};
    \node (12) at (10, -2) {};
    \node (21) at (15, 0) {};
    \node (22) at (14, -2) {};
    \node (31) at (15, 2) {};
    \node (32) at (14, 4) {};
    \node (41) at (9, 2) {};
    \node (42) at (10, 4) {};

    \draw (g1) -- (g2);
    \draw (g2) -- (g3);
    \draw (g3) -- (g4);
    \draw (g4) -- (g1);

    \draw (h1) -- (h2);
    \draw (h2) -- (h3);
    \draw (h3) -- (h1);

    \draw (1) -- (2);
    \draw (11) -- (12);
    \draw (11) -- (1);
    \draw (12) -- (1);
    
    \draw (2) -- (3);
    \draw (21) -- (22);
    \draw (21) -- (2);
    \draw (22) -- (2);
    
    \draw (3) -- (4);
    \draw (31) -- (32);
    \draw (31) -- (3);
    \draw (32) -- (3);
    
    \draw (4) -- (1);
    \draw (41) -- (42);
    \draw (41) -- (4);
    \draw (42) -- (4);

    \end{tikzpicture}
    \caption{Example of a comb product of two graphs.}
    \label{fig:combprod}
\end{figure}

The following corollary gives a linear algebraic proof of the the one-variable composition identity in \cite[Proposition~2.5]{ahl22}:

\begin{corollary}\label{staridentity}
    Let $G$ be a graph and $G'$ be a rooted graph. Then the following formula holds:
    \[
    f_{G \comb G'} (z) = f_G (- g_{G'} (z)).
    \]
\end{corollary}

\begin{proof}
    The comb product graph $G \comb G'$ is constructed by attaching a copy of $G'$ to each vertex of $G$. For each vertex of $G$, use Theorem \ref{thm_retract} to retract its copy of $G'$ into a loop with weight $g_{G} (z)$. Thus, this series of retractions results in the graph $G$ with all vertices colored with $-g_{G'} (z)$, so $f_{G \comb G'} (z) = f_G (- g_{G'} (z))$.
\end{proof}

In the two-variable case, we can restrict this product to just connecting vertices with the same loop weight, i.e. a ``compatible" comb product. Let $G$ be a graph and $G'$ a rooted graph with $z, w$ vertex coloring. The \dfn{$z$-comb product} of $G$ and $G'$, denoted $G \comb_z G'$, attaches a copy of $G'$ to each $z$-vertex of $G$. An example is provided below in Figure \ref{fig:zcomb}:

\begin{figure}[!htbp]
    \centering
    \begin{tikzpicture} [scale=.7,auto=center, every node/.style={circle,fill=blue!20}]
    \node[label=below:$z$] (g1) at (0, 0) {};
    \node[label=below:$z$] (g2) at (2, 0) {};
    \node[label=above:$z$] (g3) at (2, 2) {};
    \node[label=above:$w$] (g4) at (0, 2) {};

    \node[label=below:$z$] (h1) at (4, 0) {};
    \node[label=below:$w$] (h2) at (6, 0) {};
    \node[label=above:$z$, fill=red] (h3) at (5, 2) {};

    \node[fill=white, scale = 1] (c) at (3, 1) {$\comb_z$};
    \node[fill=white] (eq) at (7, 1) {=};

    \node[fill=red, label=below:$z$] (1) at (11, 0) {};
    \node[fill=red, label=below:$z$] (2) at (13, 0) {};
    \node[fill=red, label=above:$z$] (3) at (13, 2) {};
    \node[label=above:$w$] (4) at (11, 2) {};
    \node[label=below:$z$] (11) at (9, 0) {};
    \node[label=below:$w$] (12) at (10, -2) {};
    \node[label=below:$z$] (21) at (15, 0) {};
    \node[label=below:$w$] (22) at (14, -2) {};
    \node[label=above:$z$] (31) at (15, 2) {};
    \node[label=above:$w$] (32) at (14, 4) {};

    \draw (g1) -- (g2);
    \draw (g2) -- (g3);
    \draw (g3) -- (g4);
    \draw (g4) -- (g1);

    \draw (h1) -- (h2);
    \draw (h2) -- (h3);
    \draw (h3) -- (h1);

    \draw (1) -- (2);
    \draw (11) -- (12);
    \draw (11) -- (1);
    \draw (12) -- (1);
    
    \draw (2) -- (3);
    \draw (21) -- (22);
    \draw (21) -- (2);
    \draw (22) -- (2);
    
    \draw (3) -- (4);
    \draw (31) -- (32);
    \draw (31) -- (3);
    \draw (32) -- (3);
    
    \draw (4) -- (1);

    \end{tikzpicture}
    \caption{Example of a $z$-compatible comb product of two graphs.}
    \label{fig:zcomb}
\end{figure}
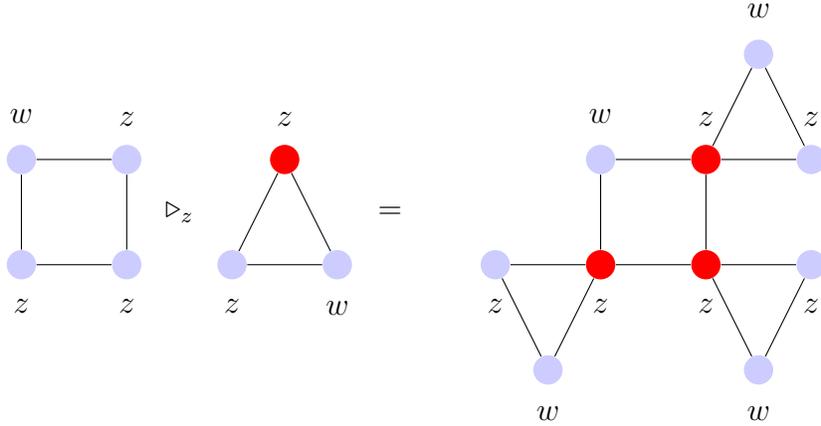

\begin{corollary}\label{combidentity}
Suppose that $G$ and $G'$ are graphs with vertex coloring $z, w$. Then
\[
f_{G \comb_z G'} (z, w) = f_G(-g_{G'}(z, w), w).
\]
\end{corollary}

\begin{proof}
    Similar to Corollary \ref{staridentity}, use Theorem \ref{thm_retract} to retract every copy of $G'$ attached to a $z$-vertex in $G$ into a loop of weight $g_G (z, w)$. Since we have only altered all $z$-colored vertices with the same function, then $f_{G \comb_z G'} (z, w) = f_G(-g_{G'}(z, w), w)$.
\end{proof}

\begin{example}

\begin{figure}[!htbp]
    \centering
    \begin{tikzpicture} [scale=.7,auto=center, every node/.style={circle,fill=blue!20}]
    \node[fill=red, label=below:$-g_{G'}$] (g1) at (8, 0) {};
    \node[fill=red, label=below:$-g_{G'}$] (g2) at (10, 0) {};
    \node[fill=red, label=above:$-g_{G'}$] (g3) at (10, 2) {};
    \node[label=above:$w$] (g4) at (8, 2) {};

    \node[fill=white] (eq) at (6.5, 1) {$\sim$};

    \node[fill=red, label=below:$z$] (1) at (1, 0) {};
    \node[fill=red, label=below:$z$] (2) at (3, 0) {};
    \node[fill=red, label=above:$z$] (3) at (3, 2) {};
    \node[label=above:$w$] (4) at (1, 2) {};
    \node[label=below:$z$] (11) at (-1, 0) {};
    \node[label=below:$w$] (12) at (0, -2) {};
    \node[label=below:$z$] (21) at (5, 0) {};
    \node[label=below:$w$] (22) at (4, -2) {};
    \node[label=above:$z$] (31) at (5, 2) {};
    \node[label=above:$w$] (32) at (4, 4) {};

    \draw (g1) -- (g2);
    \draw (g2) -- (g3);
    \draw (g3) -- (g4);
    \draw (g4) -- (g1);

    \draw (1) -- (2);
    \draw (11) -- (12);
    \draw (11) -- (1);
    \draw (12) -- (1);
    
    \draw (2) -- (3);
    \draw (21) -- (22);
    \draw (21) -- (2);
    \draw (22) -- (2);
    
    \draw (3) -- (4);
    \draw (31) -- (32);
    \draw (31) -- (3);
    \draw (32) -- (3);
    
    \draw (4) -- (1);

    \end{tikzpicture}
        \caption{Comb products and composition}
        \label{fig:combex}
\end{figure}
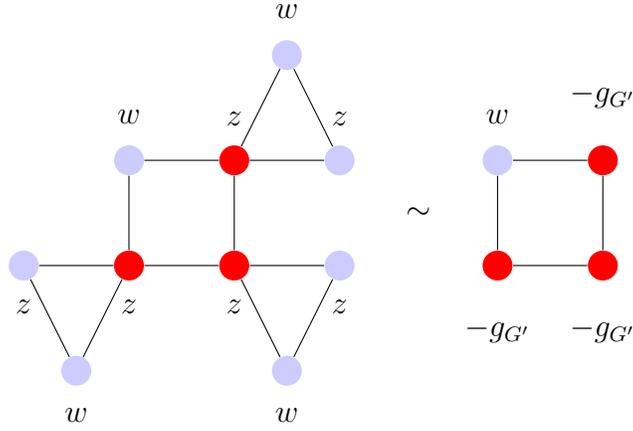
In Figure \ref{fig:zcomb}, the comb product $G \comb_z G'$ has, by the reduction theorem, the colored graph representation in Figure \ref{fig:combex} and the equivalent colored adjacency matrix
\begin{align*}
&f_{G \comb_z G'} (z, w) = \bbm -w & 1 & 1 & 0 \\ 1&  g_{G'}& 1& 0 \\ 0& 1& g_{G'}& 1 \\ 1& 0& 1& g_{G'} \ebm\inv_{(1,1)} = f_G(-g_{G'}, w) \\
&= \frac{2 (-g_{G'})-(-g_{G'})^3}{w (-g_{G'})^3-2 w (-g_{G'})-2 (-g_{G'})^2}
\end{align*}
where
\[
g_{G'}(z,w) = \frac{2 + w + 2 z - w z^2}{-1 + w z}.
\]
\end{example}


\section{Contact order}

In \cite{bps1, bps2}, Bickel, Pascoe, and Sola describe the behavior of a rational inner function near boundary singularity by considering the structure of the level curves of the function near the singularity. In particular, they prove that the branches of the level sets have local analytic parametrizations. Bickel, Pascoe and Sola work in the conformally equivalent setting of rational inner functions on the complex unit bidisk.

In this section, parallel to the work in \cite{bickelhong}, we consider functions of degree $1$ in the $w$-variable. In this case, as a main result in \cite{bickelhong}, Bickel and Hong get a refined version of Theorem \ref{blackmagicbetter} where a special singularity exists that detects this connection. Our construction produces a singularity at $(\infty, 0)$. In \cite{bickelhong}, this singularity is moved to the point $(0,0)$ with the map $z \mapsto 1/z$.

\begin{theorem}\label{blackmagicbetterbickel}[Theorem 1.1,  \cite{bickelhong}]
Let $A_G$ be the adjacency matrix of a simple undirected graph $G$ on $n$ vertices. Suppose that vertex $j$ is colored with $z$, that vertex $n$ is colored with $w$, and that vertex $j$ and $n$ are path connected.  Then
\[
f(z, w) = \mathcal (\cA_G^{-1})_{(1,1)}
\]
is a rational inner Pick function with a (regular) boundary singularity at $(\infty,0)$.
\end{theorem}

Following \cite{bps2}, this means that in a neighborhood of $\infty$, for all but a finite number of $\la \in \R$, we can solve $f(z,w) = \lambda$ to get $w = \Lambda_\la(z)$ and by \cite{bps2} this function is (locally) analytic.

\dfn{Order of contact} measures the similarity of level curves as they bunch up at a singularity - the extent to which the level curves are independent of the choice of value $\la$ can be quantified by looking at the order of vanishing of the function $\Lambda_\la(z) = \Lambda_\mu(z)$, expanded as a series in $z$ at $\infty$.

Using this idea, in \cite{bickelhong}, Bickel and Hong make the following conjecture, which they prove in several cases by way of careful analysis of Neumann series expansions and determinantal techniques. (In fact, Bickel and Hong also consider the case where the matrix $Y$ can have values of $t \in (0,1)$ on the diagonal, which leads to mixed colors on the vertices. We do not address this case here, and so that part of their conjecture remains open.)
\begin{conjecture}\label{conjecture1}
Let $G$ be a graph with a single vertex colored $w$ (say vertex $j$), and let $n$ denote the length of the shortest path connecting vertex $1$ to vertex $j$. Then the order of contact of $f(z,w) = (\A_G\inv)_{(1,1)}$ at $(\infty, 0)$ is $2n$.
\end{conjecture}

The following example illustrates the conjecture and also sheds light on the proof strategy.

\begin{example}
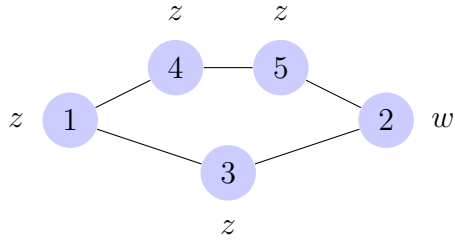
\begin{figure}[!htbp]
    \centering
    \begin{tikzpicture} [scale=.7,auto=center, every node/.style={circle,fill=blue!20}]

        \node[label=left: {$z$}] (1) at (6, 1) {$1$};
    \node[label=below:{$z$}] (3) at (9, 0) {$3$};
    \node[label=right:{$w$}] (2) at (12, 1) {$2$};
    \node[label=above:{$z$}] (4) at (8, 2) {$4$};
    \node[label=above:{$z$}] (5) at (10,2) {$5$};

    \draw (1) -- (3);
    \draw (1) -- (4);
    \draw (2) -- (3);
    \draw (4) -- (5);
    \draw (5) -- (2);
    
    \end{tikzpicture}
    \caption{Example of contact order calculation.}
    \label{fig:contact}
\end{figure}
Consider the colored graph in Figure \ref{fig:contact}, in which we will compute the representing function with respect to vertex 1. 
\[
\A_G = \begin{bNiceArray}{cc|c:cc} -z & 0 & 1 & 1 & 0 \\ 0 & -w & 1 & 0 & 1  \\\hline 1 & 1 & -z & 0 & 0  \\\hdottedline 1 & 0 & 0 & -z & 1 \\ 0 & 1 & 0 & 1 & -z
\end{bNiceArray}
\]
Notice that $2,2$ block is itself diagonal in this case (as each outbound edge from vertex 1 corresponds to precisely one inbound edge to vertex 2). Schur complements give
\begin{align*}
&(\A_G\inv)_{(1,1)} \\
&=\left( \bbm -z & 0 \\ 0 & -w \ebm - \underbrace{\bbm 1\\1 \ebm (-z)\inv \bbm 1 & 1 \ebm}_{\text{path 1 data}} - \underbrace{\bbm 1 & 0 \\ 0 & 1 \ebm \bbm -z & 1 \\ 1 & -z \ebm\inv \bbm 1 & 0 \\ 0 & 1 \ebm}_{\text{path 2 data}}\right)\inv_{(1,1)} \\
&= \left( \bbm -z & 0 \\ 0 & -w \ebm - \bbm -\frac{1}{z} & -\frac{1}{z} \\ -\frac{1}{z} & -\frac{1}{z} \ebm - \bbm -\frac{z}{z^2-1} & -\frac{1}{z^2-1} \\ -\frac{1}{z^2-1} & -\frac{z}{z^2-1}\ebm \right)\inv_{(1,1)} \\
&= \bbm -z + \frac{1}{z} + \frac{z}{z^2-1} & \frac{1}{z} + \frac{1}{z^2 - 1} \\   \frac{1}{z} + \frac{1}{z^2 - 1} & -w + \frac{1}{z} + \frac{z}{z^2 - 1} \ebm\inv_{(1,1)} \\
&= \frac{w z^3-w z-2 z^2+1}{-w z^4+3 w z^2-w+2 z^3-4 z+2}.
\end{align*}
We examine the structure of level sets near the singularity at $(\infty,0)$. First, solving
\[
\frac{w z^3-w z-2 z^2+1}{-w z^4+3 w z^2-w+2 z^3-4 z+2} = \la
\]
for $w$ to get $w = \Lambda_\la(z)$ and expanding as a series in $z$ at $\infty$ gives
\[
w = \frac{2}{z} + \frac{2}{z^3} +\frac{2}{z^4} - \frac{1}{\la z^4} + O\left(\left(\frac{1}{z}\right)^5\right).
\]
That is, the order of vanishing for $\Lambda_\la(z) - \Lambda_\mu(z)$ at $\infty$ is $4$, which is twice the length of the shortest path from vertex $1$ to vertex $3$.
\end{example}

Notice that in the example above, we have
\[ 
(\A_G\inv)_{(1,1)} = \bbm -z + g(z) & f(z) \\   f(z) & -w + g(z) \ebm\inv_{(1,1)}
\]
where $f(z)$ and $g(z)$ represent the ``path data'' of the graph. These are related to what are often called walk generating functions. Recall that if $A$ is the adjacency matrix of a graph, the entry $a^{(n)}_{ij}$ of $A^n$ counts walks of length $n$ from vertex $i$ to vertex $j$. The walk generating function is the matrix valued function
\[
W(z) = [W_{ij}(z)]_{ij} = (A - z I)\inv
\]
which will be defined off of the spectrum of $A$. Notice that for $i = j$, we recover a one-variable Nevanlinna transform $W_{ii}(z) = \langle(A- zI)\inv e_i, e_i \rangle$. While the $i\neq j$ entries of $W(z)$ need not be Pick functions, each entry will be analytic for $z$ sufficiently large. 

\begin{lemma} \label{pathfunctionlemma}
    Let $G$ be a z-colored graph with adjacency matrix A. Then the expansion at $\infty$ of the walk generating function from vertex $i$ to vertex $j$ is $-\sum_{n = 0}^{\infty} ($\# of paths from i to j of length n$)z^{-(n+1)}$.
\end{lemma}

\begin{proof}
    The walk generating function from vertex $i$ to vertex $j$ is given by $W_{ij}(z) = (A-zI)^{-1}_{(i,j)}$. But we can write, for sufficiently large $z$,
    \begin{align*}
        W_{ij}(z) = (A-zI)^{-1}_{(i,j)} &= -\dfrac{1}{z} \left(I-\dfrac{A}{z}\right)^{-1}_{(i,j)} \\
        &= -\dfrac{1}{z} \left(\sum_{n=0}^{\infty} \dfrac{A^n}{z^n}\right)_{(i,j)} \\
        &= -\sum_{n=0}^{\infty} \dfrac{(A^n)_{(i,j)}}{z^{n+1}}.
    \end{align*}
    Since the $i,j$ entry of $A^n$ gives the number of paths of length $n$ from $i$ to $j$ we have that $$W_{ij}(z) = -\sum_{n = 0}^{\infty} (\text{\# of paths from i to j of length n})z^{-(n+1)}.$$
\end{proof}

In particular, we observe that the functions $W_{ij}(z) = \langle (A - zI)\inv e_i, e_j \rangle$ admit power series expansions at $\infty$. We also note that the previous lemma can be proved with an induction on graphs, but the method is quite a bit more tedious. 

\begin{lemma} \label{pathorder}
    In a graph $G$, if the shortest path from vertex $i$ to $j$ is of length $n$ then the order of vanishing of the path function at $\infty$ from $i$ to $j$ will be $n+1$.
\end{lemma}

\begin{proof}
    Since we have an infinite series for the path function at $\infty$ we look at the first term for the order of vanishing. The first term is the first $n$ for which there is a path of length $n$ between vertices $i$ and $j$ which will be $z^{-(n+1)}$. Then the order of vanishing will be $n+1$ when the shortest path is length $n$.
\end{proof}

We are now ready to resolve Conjecture \ref{conjecture1}.

\begin{theorem}\label{contactorder}
Let $G$ be a colored graph with one vertex, say $j$, labelled $w$ and all others $z$. Then the order of contact of the Nevanlinna representation at vertex $i$ is $2n$ where the minimum path length from $i$ to $j$ is $n$.
\end{theorem}

\begin{proof}
The argument goes by cases.

\textbf{Case 1:} First, consider the case where $n \geq 2$ (that is, the vertices of interest aren't directly connected). By Lemma \ref{relabel}, without loss of generality we consider the case where we calculate the Nevanlinna representation from vertex 1 and the $w$ is located at vertex 2. Then the matrix $\A_G$ can be written
\[
\A_G = \bbm \bbm -z & 0 \\ 0 & -w \ebm & A_{12} \\ A_{21} & A_{22} \ebm,
\]
where $A_{22}$ represents the graph formed from all paths from vertices adjacent to vertex $1$ to vertices adjacent to vertex $2$, and $A_{12} = A_{21}^T$ encodes the connection of the ``middle'' graph to the terminating vertices $1$ and $2$.
Taking a Schur complement with respect to the $(1,1)$ block gets
\[
f_G(z,w) = (\A_G\inv)_{(1,1)} = \bbm -z -g(z) & -p(z) \\ -p(z) & -w - h(z) \ebm\inv_{(1,1)}
\]
where $g$ is the sum of walk generating functions that represent walks leaving and returning to vertex 1 through each adjacent vertex, $h$ is the sum of walk generating functions representing walks leaving and returning to vertex $2$ through each vertex adjacent to vertex $2$, and $p$ is the sum of walk generating functions representing the paths connecting vertex $1$ to vertex $2$ through each pair of edges adjacent to vertex $1$ and vertex $2$ respectively. We have reduced the problem to finding the $(1,1)$ entry of the inverse of a $2 \times 2$ scalar matrix. Hence,
\begin{align*}
f_G(z,w) &= \bbm -z -g(z) & -p(z) \\ -p(z) & -w - h(z) \ebm\inv_{(1,1)} \\
&= \frac{-h-w}{-p^2+g h+g w+h z+w z}.
\end{align*}
To calculate the order of contact, we find the structure of the level curves in a neighborhood of $(\infty, 0)$ by setting $$f_G(z,w) = \la.$$ Solving for $w$ gives the level curve parametrization
\[
w = \Lambda_\la(z) = \frac{p^2 \la-g h \la-h \la z-h}{g \la+\la z+1} = -h + p^2 \cdot \frac{\la}{1 + \la(g + z)}.
\]

So the analysis comes down to understanding the function
\[
\frac{\la}{1 + \la(g(z) + z)}.
\]
As $g(z)$ is a sum of walk generating functions in $z$, $g$ has a power series expansion at $\infty$, and so expanding the rational expression as a series in $z$ at $\infty$, we get
\[
\frac{\la}{1 + \la(g(z) + z)}  = \frac{1}{z}  - \frac{a_0}{z^2} - \frac{1}{\la z^2}  + O\left(\left(\frac{1}{z}\right)^3\right).
\]
That is, the first term that depends on $\la$ is order 2. Note that $p$, another sum of walk generating functions, when expanded as a power series at $\infty$, has  lowest order term of order $n-1$ where $n$ is the shortest path length between $i$ and $j$ by Lemma \ref{pathorder}. So $p^2$ has lowest order term $2n - 2$. Thus 
\[
\Lambda_\la(z) = -h + p^2 \cdot \frac{\la}{1 + \la(g + z)} = -h + \frac{p^2}{z}  - \frac{a_0 p^2 }{z^2} - \frac{p^2}{\la z^2}  + O\left(\left(\frac{1}{z}\right)^3\right),
\]
which implies that generically $\Lambda_\la(z) - \Lambda_\mu(z)$ has order of vanishing $2n$ and hence $f_G$ has order of contact $2n$ at $(\infty, 0)$. 

\textbf{Case 2:} Now consider the case where the distance from vertex 1 to vertex 2 is 1 - that is, the $z$-vertex and the $w$-vertex of the Nevanlinna representation are connected. In this case, 
\[
\A_G = \bbm \bbm -z & 1 \\ 1 & -w \ebm & A_{12} \\ A_{21} & A_{22} \ebm.
\]
Then
\begin{align*}
f_G(z,w) = (\A_G)\inv_{(1,1)} &= \bbm -z -g(z) & 1-p(z) \\ 1-p(z) & -w - h(z) \ebm\inv_{(1,1)} \\
&= \frac{-h-w}{-p^2+2 p+g h+g w+h z+w z-1}.
\end{align*}
Choosing a value $\la$ and solving for $w$ as a function of $z$ for the level curve function gives
\[
w = \Lambda_\la(z) = -h + \frac{(p-1)^2 t}{g t+t z+1}
\]
Performing a similar expansion at $\infty$ gives that $\frac{1}{\la z^2}$ is the first term depending on $\la$, and so the order of contact is $2$.

\textbf{Case 3:} Finally, consider the case where we compute a Nevanlinna representation with respect to a vertex colored $w$, enumerated as vertex 1. Then we can write
    
\[
    \A_G = \begin{bmatrix}
        w & A_{12} \\
        A_{12}^T & A_{22}
    \end{bmatrix}.
\]
Now compute the Nevalinna representation to get

    \begin{align*}
        f_G(z,w) = (\A_G\inv)_{(1,1)} &= (-w - A_{12}A_{22}^{-1}A_{12}^T)^{-1} \\
        &= \dfrac{1}{-w-p(z)}
    \end{align*}

   where $p(z)$ is the sum of walk generating functions leaving and returning to vertex $1$ through adjacent vertices. We now compute $\Lambda_\la(z)$ as above, solving $f_G(z,w) = \la$ for $w$, which gives
   \[
   \Lambda_\la(z) = -p(z) - \frac{1}{\la}.
   \]

 By Lemma \ref{pathorder}, $p(z)$ has lowest order term at least 1, and so the constant term $-\dfrac{1}{\la}$ cannot be canceled out by $p(z)$. Then the order of contact is $0$.

\end{proof}

Note that this observation gives a method for constructing families of rational Pick functions with prescribed boundary behavior at a guaranteed singularity. It remains open if we can extend this idea to design a function with more complicated prescribed boundary behavior. We should note that generally this construction creates two or more singularities in the resulting function - one remaining question is to understand additionally arising boundary singularities in terms of the graph structure.

\section{A combinatorial observation}
The following observation of potentially independent interest arose in the development of the function theory of the previous section. Following Bickel and Hong's terminology \cite{hong}, a \dfn{stick graph} (more commonly called  a \textit{path graph} in graph theory)  consists of a line of vertices with adjacent vertices connected by edges. Let $S_n$ be the $z$-colored adjacency matrix of the n-stick, a stick with $n$ vertices. For example:
    \begin{align*}
        S_1 &= \begin{bmatrix}
            -z
        \end{bmatrix} \\
        S_2 &= \begin{bmatrix}
            -z & 1 \\ 1 & -z
        \end{bmatrix} \\
        S_3 &= \begin{bmatrix}
            -z & 1 & 0 \\ 1 & -z & 1 \\ 0 & 1 & -z
        \end{bmatrix}
    \end{align*}
    and so on.
    Let $T_n = \det(S_n)$. The following observation may be of independent interest. (See also \cite[Section 3.1]{rahm}.)

    \begin{theorem}
        For stick graphs the determinant $T_n$, of stick $S_n$, is the coefficient of $x^n$ in the power series of 
        
        \[f(x) = \frac{1}{1+x^2+xz}.\]
    \end{theorem}
    \begin{proof}
        Note that, \begin{align*}
            T_n &= \det(S_n) = \det \left(\left[\begin{array}{c|c}
    -z & \begin{matrix} 1 & 0 & \cdots & 0 \end{matrix} \\
    \hline \begin{matrix}
         1 \\ 0 \\ \vdots \\ 0
     \end{matrix}& S_{n-1}
    \end{array}\right]\right) \\
            &= \det \left(\left[\begin{array}{c|c}
    \begin{matrix} -z & 1 \\ -z & 1 \end{matrix} &
    \begin{matrix}0 & \cdots & & 0\\ 1 & 0 & \cdots & 0\end{matrix}\\ \hline
     \begin{matrix} 0&1 \\ \vdots&0 \\ & \vdots\\ 0&0\end{matrix} & S_{n-2}
    \end{array}\right]\right)\\
            &= -z \cdot \det(S_{n-1}) - 1 \cdot \det \left(\left[\begin{array}{c|c}
                1 & \begin{matrix} 0 & \cdots & 0 \end{matrix} \\ \hline
                \begin{matrix} 1 \\ 0 \\ \vdots \\ 0  \end{matrix} & S_{n-2}
            \end{array}\right]\right)\\
            &= -z \cdot \det(S_{n-1}) - \det(S_{n-2}) \\
            &= -z \cdot T_{n-1} - T_{n-2}
        \end{align*}

        With this recurrence relation and the fact that $T_1 = -z$, $T_2 = z^2 - 1$, we can solve this for a generating function. Noting that $T_0 = 1$ fits the recurrence relation, we solve for the ordinary generating function of $T_n$, $f(x) = \sum\limits_{n=0}^\infty T_n x^n$.

        \begin{align*}
            T_n &= -z \cdot T_{n-1} - T_{n-2}, \qquad  n\ge 2\\
            \sum\limits_{n=2}^\infty T_n x^n &= -z \cdot \sum\limits_{n=2}^\infty T_{n-1}x^n - \sum\limits_{n=2}^\infty T_{n-2}x^n \\
            \sum\limits_{n=2}^\infty T_n x^n &= -zx \sum\limits_{n=1}^\infty T_n x^n - x^2 \sum\limits_{n=0}^\infty T_n x^n \\
            f(x) - xT_1 - T_0 &= -zx(f(x) - T_0) - x^2f(x) \\
            f(x)(1 + zx + x^2) &= -zx + 1 + zx \\
            f(x) &= \frac{1}{1 + zx + x^2}
        \end{align*}
    \end{proof}

The polynomial coefficients of the series representation of $\frac{1}{1 + zx + x^2}$ have numerical coefficients that are certain diagonals in Pascal's triangle and arise in certain so-called Riordan arrays (see \cite[A109466]{oeis} and  \cite{barry}). 

\section{Open Questions}

Certain obvious extensions remain. In addition to the set of open questions in the final section of \cite{bickelhong}, we have computational evidence for the following conjecture.
\begin{conjecture}
    Given a graph $G$, colored with $z$ and $w$, the contact order of the representation computed from vertex $i$ is twice the path distance to the closest vertex colored $w$.
\end{conjecture}
    
This generalizes Theorem \ref{contactorder} to include graphs where there is not just one but multiple $w$ vertices, which results in rational inner Pick functions of general degree $(n,m)$. The difficulty is in the local analytic parametrization of the level curves near a singularity, which will have branches rather than a single function \cite{bps2}. However, evidence suggests that every branch has the exact same order of contact. 

We also have some preliminary work on what happens if vertices are colored with convex combinations of $z$ and $w$. Bickel and Hong consider this case. While the star and comb product formulas carry through in this setting, we have not been able to use our methods to explore questions about singular behavior.

More generally, there seem to be connections between this graph-theoretic interpretation of the Nevanlinna representation and new work by one of the authors and J. E. Pascoe on the representation of formal languages by noncommutative series \cite{combconv} that would be worth exploring.

\section*{Acknowledgments}

The authors would like to thank Jeff Liese for a key observation connecting walk generating functions to the Nevanlinna representation. The authors also thank Kelly Bickel for inspiring this project, as well as comments and conversations about the connections between her work and ours. Additionally, this work would not be possible without extensive conversations about the connections between graph combinatorics and realizations with J. E. Pascoe. Finally, we wish to acknowledge the support of the Cal Poly Frost Summer Research Program.

\bibliography{references}
\bibliographystyle{plain}

\end{document}